\documentclass[11pt]{amsart}
\usepackage[latin1]{inputenc}
\usepackage{amsmath}
\usepackage{amsfonts}
\usepackage{amssymb}
\usepackage{graphicx}
\usepackage{fourier}
\usepackage{dsfont}
\usepackage{hyperref}
\usepackage{enumerate}

\newtheorem{theorem}{Theorem}[section]
\newtheorem{lemma}[theorem]{Lemma}

\theoremstyle{definition}
\newtheorem{definition}[theorem]{Definition}

\newtheorem{corollary}[theorem]{Corollary}
\theoremstyle{remark}
\newtheorem{remark}[theorem]{Remark}

\numberwithin{equation}{section}

\author{Amit Einav}
\title{A Few Ways to Destroy Entropic Chaoticity on Kac's Sphere.}
\thanks{The author was supported by ERC grant MATKIT}

\begin{document}

\maketitle

\begin{abstract}
In this work we discuss a few ways to create chaotic families that are not entropically chaotic on Kac's Sphere. We present two types of examples: limiting convex combination of an entropically chaotic family with a particularly 'bad' non-entropic family, and two explicitly computable families that vary rapidly with $N$, causing loss of support on the sphere or high entropic tails. 
\end{abstract}

\section{Introduction}\label{sec: introduction}
In his 1956 paper, \cite{Kac}, Kac introduced the concept of chaotic families (or 'The Boltzmann property' in his words) as a condition on the initial data to the solution of his many-particle, binary collision, stochastic process, from which a caricature of Boltzmann's equation arises. Motivated by Boltzmann's 'Stosszahlansatz' assumption, stating that pre-collision particles can be considered to be independent, Kac defined the chaoticity of a family $\left\lbrace F_N \right\rbrace_{N\in\mathbb{N}}$ of probability densities on the sphere $\mathbb{S}^{N-1}\left(\sqrt{N} \right)$ as:
\begin{definition}\label{def: chaoticity}
A sequence of symmetric probability densities, $\left\lbrace F_N \right\rbrace_{N\in\mathbb{N}}$, on the sphere $\mathbb{S}^{N-1}\left(\sqrt{N} \right)$ is said to be $f-$chaotic if there exists a probablity density, $f$, such that 
\begin{equation}\label{eq: chaoticity}
\lim_{N\rightarrow\infty}\Pi_k(F_N)(v_1,\dots,v_k)=f^{\otimes k}(v_1,\dots,v_k)
\end{equation}
for every $k\in\mathbb{N}$, where $\Pi_k(F_N)$ is the $k-$th marginal of $F_N$ and the limit is taken in the weak topology induced by bounded continuous functions on $\mathbb{R}^k$.
\end{definition} 
In what follows we will use the term 'Kac's sphere' (or 'the sphere' when context permits) for $\mathbb{S}^{N-1}\left(\sqrt{N} \right)$. The fact that we deal with a sphere of radius $\sqrt{N}$ is crucial to the process, and quite intuitive. Indeed, if we're talking about a process involving $N$ particles with one dimensional velocities, each indistinguishable from the other, then assuming that a particle (and thus every particle) has a unit of energy leads to the conclusion that the total energy of the system is $N$ units. By conservation of energy, which Kac's model satisfies, the whole system must be restricted to the sphere.\\
Definition (\ref{def: chaoticity}) can easily be extended to general measures on the sphere. Indeed, we only need to define what it means to be symmetric. 
\begin{definition}\label{def: symmetry of measures}
A measure $\mu_N$ on Kac's sphere is called symmetric if for any measurable function $F_N$ and for any permutation $\tau\in S_N$ we have that
\begin{equation}\label{eq: symmetry of measures}
\int_{\mathbb{S}^{N-1}\left( \sqrt{N} \right)}F_N(v_1,\dots,v_N)d\mu_N=\int_{\mathbb{S}^{N-1}\left( \sqrt{N} \right)}F_N\left(v_{\tau(1)},\dots,v_{\tau(N)}\right)d\mu_N.
\end{equation}
\end{definition}
Kac considered a model in which $N$ indistinguishable particles, with one dimensional velocities, underwent random binary collisions. His evolution equation for the probability density of the velocities of the particles was given by
\begin{equation}\label{eq: master equation}
\frac{\partial F_N}{\partial t}(v_1,\dots ,v_N)=-N(I-Q) F_N(v_1,\dots ,v_N),
\end{equation}
where
\begin{equation}\label{eq: def of Q}
\begin{gathered}
QF\left(v_1,\dots,v_N \right)=\frac{1}{2\pi}\cdot \frac{2}{N(N-1)}\cdot \\
\sum_{i<j}\int_{0}^{2\pi}F\left(v_1,\dots,v_i(\vartheta),\dots,v_j(\vartheta),\dots,v_N\right)d\vartheta ,
\end{gathered}
\end{equation}
with 
\begin{equation}\label{eq: v_i(theta),v_j(theta)}
\begin{gathered}
v_i(\vartheta)=v_i\cos(\vartheta)+v_j\sin(\vartheta), \\
v_j(\vartheta)=-v_i\sin(\vartheta)+v_j\cos(\vartheta).
\end{gathered}
\end{equation}
Kac managed to show that chaoticity is the right ingredient to derive Boltzmann's equation from his linear $N-$particle model. He managed to show that (\ref{eq: chaoticity}) \emph{propagates in time} under his evolution equation, and that the evolution equation for the limit probability density, $f(v,t)$, satisfies a caricature of Boltzmann's equation. Kac expressed hope that investigating his $N-$particle linear model would lead to new results on the Boltzmann's equation, particularly in the area of trend to equilibrium. Indeed, It is easy to see that $Q$ is bounded and self adjoint on Kac's sphere as well as $Q<I$. The ergodicity of (\ref{eq: master equation}) leads to the fact that for every fixed $N$ we have that $\lim_{t\rightarrow\infty}F_N(v_1,\dots,v_N,t)=1$. Defining the spectral gap
\begin{equation}\label{eq: spectral gap}
\Delta_N=\inf \left\lbrace \frac{\left\langle F_N,N(I-Q) F_N \right\rangle}{\left\lVert F_N \right\rVert_{L^2\left(\mathbb{S}^{N-1}\left(\sqrt{N} \right) \right)}} \quad F_N\perp 1 \right\rbrace,
\end{equation}
one can show that if $F_N(t)=F_N(v_1,\dots,v_N,t)$ solves (\ref{eq: master equation}) then:
\begin{equation}\label{eq: spectral gap relaxation}
\left\lVert F_N(t)-1 \right\rVert_{L^2\left(\mathbb{S}^{N-1}\left(\sqrt{N} \right) \right)}\leq e^{-\Delta_N t}
\left\lVert F_N(0)-1 \right\rVert_{L^2\left(\mathbb{S}^{N-1}\left(\sqrt{N} \right) \right)}.
\end{equation}
Kac conjectured that $\liminf_{N\rightarrow\infty}\Delta_N>0$ and hoped that it will lead to an exponential rate of decay for Boltzmann's equation as a limit equation of his linear model. While the conjecture was proven to be true (see \cite{CCL, CGL, Jan, Maslen}) the choice of $L^2$ as a reference distance is catastrophic when considering chaotic families. Intuitively speaking, one would suspect that chaoticity means (in some sense) that $F_N\approx f^{\otimes N}$. As such, we will have that the $L^2$ norm of $F_N$ will be exponentially large. Indeed, one can easily construct a chaotic family $F_N(0)$ with $\left\lVert F_N(0) \right\rVert_{L^2 \left(\mathbb{S}^{N-1}\left(\sqrt{N} \right) \right)}\geq C^N$, where $C>1$, leading to a relaxation time that is proportional to $N$.\\
A different approach, one more amiable to chaoticity, was needed. A natural quantity to investigate, one that was investigated by Boltzmann himself in his famous $H-$theorem, is the entropy. In Kac's context the entropy is defined as 
\begin{equation}\label{eq: entropy}
H_N(F_N)=\int_{\mathbb{S}^{N-1}\left(\sqrt{N} \right)} F_N \log F_N d\sigma^N,
\end{equation}
where $d\sigma^N$ is the uniform probability measure on the sphere. This is a particular case of the \emph{relative entropy} between two probability measures, defined as:
\begin{definition}\label{def: rel entropy}
Given two probability measure, $\mu$ and $\nu$, we define the relative entropy
\begin{equation}\label{eq: rel entropy}
H(\mu | \nu)=\int f\log f d\nu,
\end{equation}
where $f=\frac{d\mu}{d\nu}$, when $\mu \ll \nu$ and $H(\mu | \nu)=\infty$ otherwise. 
\end{definition}
The relative entropy has some useful properties. In our context, the most important one is the Csiszar-Kullback-Leibler-Pinsker inequality:
\begin{equation}\label{eq: Kullback inequality}
\left\lVert \mu-\nu \right\rVert^2_{TV}\leq 2H(\mu | \nu),
\end{equation}
giving us a way to measure distance between measures (and in particular between probability densities). Notice that much like the log-Sobolev inequality, the constant appearing in (\ref{eq: Kullback inequality}) is \emph{independent of the dimension}, giving us a way to uniformly control the distance!\\
By definition $H_N(F_N)=H\left(F_N d\sigma^N | d\sigma^N \right)$, and as such 
\begin{equation}\label{eq: distance with rel entropy}
\int \left\lvert F_N-1\right\rvert d\sigma^N \leq \sqrt{2H_N(F_N)},
\end{equation} 
so the entropy can serve as a tool to measure convergence in Kac's context.\\
Another very appealing property of the entropy is its \emph{extensivity}. Due to the properties of the logarithm one can hope that if $F_N$ is $f-$chaotic then, in some way,
\begin{equation}\label{eq: entropic motivation}
H_N(F_N)\approx N\cdot H(f | \gamma),
\end{equation}
where $\gamma(x)=\frac{e^{-\frac{x^2}{2}}}{\sqrt{2\pi}}$ is the standard Gaussian (the appearance of the Gaussian shouldn't be too surprising - it is a known fact that the uniform measure on Kac's sphere is $\gamma-$chaotic!).\\
At this point one can define a 'spectral gap' for the entropy, and see if it yields better results than the linear theory. Assuming that $F_N$ is a symmetric probability density that solves (\ref{eq: master equation}) one can define 
\begin{equation}\label{eq EEP-ratio}
\begin{gathered}
\Gamma_N=\inf \left\lbrace \frac{\left\langle N(I-Q) F_N,\log F_N\right\rangle}{H_N(F_N)} \right\rbrace,
\end{gathered}
\end{equation}
and conclude that 
\begin{equation}\label{eq: EEP-ratio distance}
H_N(F_N(t))\leq e^{-\Gamma_N t}H_N(F_N(0)).
\end{equation}
If $\Gamma_N>C>0$ for all $N$ we can combine (\ref{eq: EEP-ratio distance}) with (\ref{eq: Kullback inequality}) and (\ref{eq: entropic motivation}) and get relaxation time that is proportional to $\log N$, which is a fantastic result. The conjecture of the existence of such constant is called 'The many-particle Cercignani's Conjecture', following a similar conjecture for Boltzmann's equation (see \cite{Cerc}) trying to find a constant $C>0$ such that
\begin{equation}\label{eq: Cerc conj}
-\frac{d}{dt}H(f(t)) \geq C H(f(t)),
\end{equation}
where $f(t)$ is the solution to Boltzmann's equation. Unfortunately, if we impose no restrictions on the probability densities the conjecture is not true and in fact $\Gamma_N\approx \frac{1}{N}$, putting us in the same place as the linear spectral gap (see \cite{Villani, Einav1, Einav2}). This obviously leads to many very interesting questions about possibilities of the conjecture being true under plausible restrictions on $F_N$.\\
While Kac's model is a big step forwards in Kinetic Theory, it had some flaws. The model was one dimensional, and as such couldn't conserve energy and momentum at the same time. Another problem with the model was the simplistic collision kernel and the inability to deal with physical kernels, depending on the velocities of the particles. In 1967 McKean extended the model to the case where the velocities were $d-$dimensional, with $d>1$, and showed that, similar to the original model, the real Boltzmann equation arises from it in an extended array of collisional kernels (see \cite{McKean}), though the restriction that the kernel would be independent of the velocities was still imposed, leaving the interesting cases of Hard Spheres and True Maxwellian Molecules unsolved.\\
In a remarkable recent paper, \cite{MM}, Mischler and Mouhot introduced a new abstract method that allowed them to tackle many unsolved questions in the subject, including the velocity dependent cases mentioned above. They managed to show quantitative and uniform in time propagation of chaos in weak measure distance, propagation of entropic chaos (soon to be defined) and quantitative estimation on relaxation rates that are \emph{independent of the number of particles}. There is more to be said and explored in the subject, but their work is a huge leap forward in the desired direction.\\
At this point we will leave Kac's models and program aside, and concentrate on the problem we wish to deal with. More information about the topic and the related spectral gap problem and entropy-entropy production ratio can be found in \cite{CCL, CCL1, CCRLV, CGL} and the excellent \cite{VReview, MM}. \\
We start by defining the concept of entropic chaoticity.
Motivated by (\ref{eq: entropic motivation}) we introduce the following, more general, definition:
\begin{definition}\label{eq: entropic chaoticity}
A family of symmetric probability measures, $\left\lbrace \mu_N \right\rbrace_{N\in\mathbb{N}}$, on Kac's sphere is said to be entropically chaotic if it is $\mu-$chaotic and 
\begin{equation}\label{eq: entropic chaoticity definition}
\lim_{N\rightarrow\infty}\frac{H(\mu_N | d\sigma^N)}{N}=H(\mu|\gamma).
\end{equation}
\end{definition}
The above definition was introduced by Carlen, Carvalho, Le Roux, Loss and Villani in \cite{CCRLV}. The authors noted that the concept of entropic chaoticity is stronger than that of mere chaoticity as it involves \emph{all} of the variables, and not just a finite amount of them. We refer the reader to \cite{CCRLV} for more interesting details, and beautiful results, about entropic chaoticity. The case where $H(\mu | \gamma)=\infty$ is somewhat of a pathological case and so in the following we will only talk about cases where $H(\mu | \gamma)$ is finite. \\
It is worth noting that in his original paper (\cite{Kac}) Kac was aware of the extensivity property of the entropy, and while he didn't define entropic chaoticity, he figured it will play an important role in his model (he thought that it will help establish a satisfactory derivation of Boltzmann's $H-$theorem). \\
In our paper, we will be solely interested in the 'functional' case where $\mu_N=F_N d\sigma^N$ and $\mu=f(x)dx$.\\ 
At this point one might ask oneself - Are there any chaotic and/or entropically chaotic families? A partial solution to this question was already given by Kac in \cite{Kac}: He noted that probability densities of the form 
\begin{equation}\label{eq: usual suspect}
F_N(v_1,\dots,v_N)=\frac{\prod_{i=1}^N f(v_i)}{\int_{\mathbb{S}^{N-1}\left(\sqrt{N} \right)} \prod_{i=1}^N f(v_i) d\sigma^N}
\end{equation}
are $f-$chaotic under some severe conditions on $f$ (very strong integrability conditions). 
Note that this type of family seems very reasonable - intuitively speaking it is an independent family on the entire space which is being restricted to the sphere, causing some (hopefully small in the limit) correlations to appear.\\
In \cite{CCRLV} the authors have managed to significantly extend Kac's result:
\begin{theorem}\label{thm: usual suspects are good}
Let $f$ be a probability density on $\mathbb{R}$ such that $f\in L^p(\mathbb{R})$ for some $p>1$, $\int_\mathbb{R} x^2 f(x)=1$ and $\int_\mathbb{R} x^4 f(x)dx<\infty$. Then the family of densities defined in (\ref{eq: usual suspect}) is $f-$chaotic. Moreover, it is $f-$entropically chaotic.
\end{theorem}
Recently, Carrapatoso has extended this result to the more realistic McKean model, conditioned to the Boltzmann sphere instead of the Kac's sphere (see \cite{Carr}).\\
As we saw before, entropic chaoticity is a very intuitive concept that arises naturally when one investigate relationships between the relaxation rates to equilibrium in the $N$-particle model and its mean field limit. We would like to understand the concept better and explore the delicate balance required for entropic chaoticity to hold. In order to do that, we explore in this paper ways to construct families of probability densities that are chaotic but \emph{not} entropically chaotic, noting the reasons for that. Our first result is the following:
\begin{theorem}\label{thm: limiting convex combination result}
Let $f$ satisfy the conditions of Theorem \ref{thm: usual suspects are good}, then there exists an $f-$chaotic family that is not entropically chaotic.
\end{theorem}
The method to prove this theorem is one of a limiting convex combination, and would be described in Section \ref{sec: limiting convex}. This is not the only way to destroy chaoticity. A different way is to create families that depend on $N$ strongly, and not only as an increase of the number of variable. Our next two results will deal with two explicitly computable family of probability densities that fails entroic chaoticity due to that reason.
\begin{theorem}\label{thm: entropic fail on usual suspect}
Let $f_N(v)=\delta_N M_{\frac{1}{2\delta_N}}(v)+(1-\delta_N) M_{\frac{1}{2(1-\delta_N)}}(v)$ where $M_a (v)=\frac{e^{-\frac{v^2}{2a}}}{\sqrt{2\pi a}}$ and $\delta_N=\frac{1}{N^\eta}$ with $\eta$ close to $1$. Then the family of probability densities defined in (\ref{eq: usual suspect}) is $M_{\frac{1}{2}}-$chaotic but not entropically chaotic.
\end{theorem}
We will see that the reason behind this failure is that the rapid change of $N$ causes the family to 'lose support at infinity'. The last result we will show is the following:
\begin{theorem}\label{thm: entropci fail on polynomial}
Let $F_N=\frac{\sum_{i=1}^N |v_i|^N}{\mathfrak{Z}_N}$ where $\mathfrak{Z}_N$ is the appropriate normalization function. Then $\left\lbrace F_N \right\rbrace_{N\in\mathbb{N}}$ is $M_{\frac{1}{2}}-$chaotic but not entropically chaotic.
\end{theorem}
The reason behind this failure will be too high an entropic tail.\\ 
The paper is structured as follows: Section \ref{sec: limiting convex} will describe the idea of limiting convex combination and will show how exactly such idea will be useful in building chaotic families that are not entropically chaotic. Sections \ref{sec: concentration} and \ref{sec: stereographic} will apply that idea to build our first two examples. The first using concentration methods with the natural coordinates on the sphere and the second using the stereographic projection and a process of 'pushing' the function to 'infinity'. Section \ref{sec: marginals on the sphere} will provide a few technical lemmas that will help us with explicit computation on the sphere, while in Section \ref{sec: escaping tensorisation} we will prove Theorem \ref{thm: entropic fail on usual suspect}. In Section \ref{sec: polynomials} we will prove Theorem \ref{thm: entropci fail on polynomial}  as well as introduce another family of polynomials that is entropically chaotic (to stress the effect of the varying power). Lastly, in Section \ref{sec: final remarks} we will discuss a few closing remarks. The Appendix to the paper contains more detailed information about the stereographic projection we use.\\
\textbf{Acknowledgement:}
The author would like to thank Cl\'ement Mouhot for many fruitful discussions, sharing of ideas and constant encouragement as well as careful reading of the manuscript and providing many useful remarks.
\section{Limting Convex Combinations.}\label{sec: limiting convex}
The concept of convexity is not alien to that of chaoticity or entropy. Several counter examples to known conjectures (such as Cercignani's conjecture) have been built using a convex combination of special stationary states (see \cite{BC}). Recently, the author has used a similar idea, but with convex coefficients that depend on $N$, in order to find an explicit bound to the entropy-entropy production ratio (see \cite{Einav1, Einav2}) - this idea is behind what we will call 'limiting convex combination'  
\begin{definition}\label{def: limit convex}
Let $\left\lbrace G_N \right\rbrace_{N\in\mathbb{N}}$ and $\left\lbrace F_N \right\rbrace_{N\in\mathbb{N}}$ be families of probability densities on $\mathbb{S}^{N-1}\left(\sqrt{N} \right)$ and let $\left\lbrace \alpha_N \right\rbrace_{N\in\mathbb{N}}$ be a sequence of real numbers such that $0<\alpha_N<1$ for all $N\in\mathbb{N}$, and $\lim_{N\rightarrow\infty}\alpha_N=0$. Then the family of probability densities
\begin{equation}\label{eq: limit convex}
C_N=(1-\alpha_N)G_N+\alpha_N F_N,
\end{equation}
is called the limiting convex combination of $G_N$ and $F_N$.
\end{definition}
We will start with a few simple properties of the limiting convex combination.
\begin{lemma}\label{lem: limting convex retain chaoticity}
Let $\left\lbrace G_N \right\rbrace_{N\in\mathbb{N}}$ and $\left\lbrace F_N \right\rbrace_{N\in\mathbb{N}}$ be symmetric probability densities on $\mathbb{S}^{N-1}\left(\sqrt{N} \right)$. If $\left\lbrace G_N \right\rbrace_{N\in\mathbb{N}}$ is $g-$chaotic then any limiting convex combination of $G_N$ and $F_N$ is $g-$chaotic. 
\end{lemma}
\begin{proof}
Assume $C_N$ is a limiting convex combination as defined in (\ref{eq: limit convex}). Given any $\phi\in C_b \left( \mathbb{R}^k \right)$, for a fixed $k\in\mathbb{N}$, we have that
\begin{equation}\label{eq: alpha_N F_N doesn't ruin chaoticity}
\begin{gathered}
\left\lvert \alpha_N \int_{\mathbb{S}^{N-1}\left(\sqrt{N} \right)} F_N(v_1,\dots,v_N)\phi(v_1,\dots,v_k)d\sigma^N \right\rvert
\leq \alpha_N \left\lVert \phi \right\rVert_\infty\underset{N\rightarrow\infty}{\longrightarrow}0.
\end{gathered}
\end{equation}
And
\begin{equation}\label{eq: 1-alpha_N G_N is chaotic}
\begin{gathered} 
(1-\alpha_N) \int_{\mathbb{S}^{N-1}\left(\sqrt{N} \right)} G_N(v_1,\dots,v_N)\phi(v_1,\dots,v_k)d\sigma^N\\ \underset{N\rightarrow\infty}{\longrightarrow}\int_{\mathbb{R}^k} g^{\otimes k}(v_1,\dots,v_k)\phi(v_1,\dots,v_k)dv_1 \dots dv_k,
\end{gathered}
\end{equation}
proving the result.
\end{proof}
\begin{remark}\label{rem: chaoticity is weak}
Notice that in Lemma \ref{lem: limting convex retain chaoticity} there is no requirement of chaoticity on $F_N$, only that of symmetry! This shows how weak the condition of chaoticity is with respect to limiting convex combination.
\end{remark}
What of entropic chaticity? Can we get any result similar to our previous lemma? The answer to this question is \emph{Yes}, but more than that - we can find simple conditions when limiting convex combinations are \emph{not} entropically chaotic.
\begin{lemma}\label{lem: limit convex and entropic chaoticity}
Let $\left\lbrace G_N \right\rbrace_{N\in\mathbb{N}}$ be a $g-$entropically chaotic family of probability densities and $\left\lbrace F_N \right\rbrace_{N\in\mathbb{N}}$ be symmetric probability densities on Kac's sphere. Then 
\begin{enumerate}[(i)]
\item If $\limsup_{N\rightarrow\infty} \frac{H_N(F_N)}{N}<\infty$ then any limiting convex combination of $G_N$ and $F_N$ is $g-$entropically chaotic.
\item If $\liminf_{N\rightarrow\infty} \frac{H_N(F_N)}{N}=\infty$ then there exists a limiting convex combination of $G_N$ and $F_N$ that is not $g-$entropically chaotic but is $g-$chaotic.
\end{enumerate}
\end{lemma}
\begin{corollary}\label{cor: limit convex of entropically chaotic is entropically chaotic}
If $G_N$ and $F_N$ are entropically chaotic then so is any limiting convex combination of them. 
\end{corollary}
\begin{proof}[Proof of Lemma \ref{lem: limit convex and entropic chaoticity}]
The $g-$chaoticity of any limiting convex combination was established in Lemma \ref{lem: limting convex retain chaoticity} so we only need to check the additional condition of entropic chaos.
\begin{enumerate}[(i)]
\item Since the function $H(x)=x\log x$ is convex we find that
\begin{equation}\label{eq: limit convex retain entropy I}
H_N\left(C_N \right)\leq (1-\alpha_N)H_N(G_N)+\alpha_N H_N(F_N).
\end{equation}
Thus
\begin{equation}\label{eq: limit convex retain entropy II}
\limsup_{N\rightarrow\infty}\frac{H_N\left(C_N \right)}{N}\leq H(g|\gamma)+\limsup_{N\rightarrow\infty}\alpha_N \cdot\frac{H_N(F_N)}{N}=H(g|\gamma).
\end{equation}
On the other hand, since $C_N$ is $g-$chaotic we have that 
\begin{equation}\label{eq: limit convex retain entropy III}
H(g|\gamma)\leq \liminf_{N\rightarrow\infty}\frac{H_N\left(C_N \right)}{N}
\end{equation}
(see \cite{CCRLV} for the proof). Combining (\ref{eq: limit convex retain entropy II}) and (\ref{eq: limit convex retain entropy III}) yields the desired result.
\item Since the logarithm is an increasing function, and $F_N$ and $G_N$ are non negative we find that
\begin{equation}\label{eq: limit convex destroys entropy I}
\begin{gathered}
H_N\left(C_N \right)=(1-\alpha_N)\int_{\mathbb{S}^{N-1}\left(\sqrt{N} \right)}G_N \log\left((1-\alpha_N)G_N+\alpha_N F_N \right)d\sigma^N\\ + 
\alpha_N\int_{\mathbb{S}^{N-1}\left(\sqrt{N} \right)}F_N \log\left((1-\alpha_N)G_N+\alpha_N F_N \right)d\sigma^N
\\ \geq (1-\alpha_N)\int_{\mathbb{S}^{N-1}\left(\sqrt{N} \right)}G_N \log\left((1-\alpha_N)G_N \right)d\sigma^N\\ 
+ \alpha_N\int_{\mathbb{S}^{N-1}\left(\sqrt{N} \right)}F_N \log\left(\alpha_N F_N \right)d\sigma^N\\ 
=(1-\alpha_N)\log\left(1-\alpha_N \right)+(1-\alpha_N)H_N(G_N)+\alpha_N \log \alpha_N + \alpha_N H_N(F_N).
\end{gathered}
\end{equation}
Thus,
\begin{equation}\label{eq: limit convex destroys entropy II}
\begin{gathered}
\liminf_{N\rightarrow\infty}\frac{H_N\left(C_N \right)}{N} \geq  
H(g|\gamma)+ \liminf_{N\rightarrow\infty}\alpha_N \cdot \frac{H_N(F_N)}{N}.
\end{gathered}
\end{equation}
Since $\liminf_{N\rightarrow\infty}\frac{H_N(F_N)}{N}=\infty$ we can easily pick $\alpha_N$ such that $\liminf_{N\rightarrow\infty}\frac{H_N(C_N)}{N}>C$ for any $C>0$, as well as $C=\infty$. This completes the proof.
\end{enumerate}
\end{proof}
Lemma \ref{lem: limit convex and entropic chaoticity} gives us the tool to find chaotic families that are not entropically chaotic: we only need to find a family of symmetric probability densities $\left\lbrace F_N \right\rbrace_{N\in\mathbb{N}}$ such that $\liminf_{N\rightarrow\infty}\frac{H_N(F_N)}{N}=\infty$. That is exactly what we will do in the following two section. This allows us to prove Theorem \ref{thm: limiting convex combination result}:
\begin{proof}[Proof of Theorem \ref{thm: limiting convex combination result}]
This immediate from Lemma \ref{lem: limit convex and entropic chaoticity} and Theorem \ref{thm: usual suspects are good}.
\end{proof}

\section{First Example: Concentration.}\label{sec: concentration}
Motivated by Lemma \ref{lem: limit convex and entropic chaoticity} and ideas of concentration in \cite{BC}, we now construct the first family of symmetric probability measures on the sphere that has entropic rate of increase that is greater than a linear one. In order to do that we will use the natural coordinates on the sphere.\\
The surface element of a sphere in $\mathbb{R}^k$ with radius $R$, expressed with its spherical angles, $\theta,\phi_1,\dots,\phi_{k-1}$, is given by
\begin{equation}\label{eq: natural sphere coordinates}
ds^{k}_R=kR^{k-1}\sin^{k-2}(\phi_1)\sin^{k-3}(\phi_2)\dots\sin(\phi_{k-1}).
\end{equation}
In particular, if we integrate over a function depending only on the elevation angle, $\phi_1$, we find that
\begin{equation}\label{eq: natural coordinates spherical integration I}
\begin{gathered}
\int_{\mathbb{S}^{k-1}(R)}g(\phi_1)d\sigma=\frac{k}{\left\lvert \mathbb{S}^{k-1} \right\rvert}\\
\cdot\int_0^{2\pi} \int_0^\pi \dots \int_0^\pi g(\phi_1)\sin^{k-2}(\phi_1)\sin^{k-3}(\phi_2)\dots\sin(\phi_{k-1})d\theta d\phi_1\dots d\phi_{k-1}.
\end{gathered}
\end{equation}
Using the formula 
\begin{equation}\label{eq: Beta formula}
B(\xi,\zeta)=2\int_0^{\frac{\pi}{2}}\sin^{2\xi-1}(\theta)\cos^{2\zeta-1}(\theta)d\theta
\end{equation} 
we find that 
\begin{equation}\label{eq: natural coordinates spherical integration II}
\begin{gathered}
\int_0^{\pi}\sin^{k-2}(\phi)d\phi=2\int_0^{\frac{\pi}{2}}\sin^{2\cdot\left(\frac{k-1}{2}\right)-1}(\phi)\cos^{2\cdot\frac{1}{2}-1}d\phi=B\left(\frac{k-1}{2},\frac{1}{2} \right)\\
=\frac{\Gamma \left(\frac{k-1}{2} \right)\sqrt{\pi}}{\Gamma \left(\frac{k}{2} \right)},
\end{gathered}
\end{equation}
leading to
\begin{equation}\label{eq: integration on sphere with angle coordinates}
\int_{\mathbb{S}^{k-1}(R)}g(\phi_1)d\sigma = \frac{\Gamma \left(\frac{k}{2} \right)}{\Gamma \left(\frac{k-1}{2} \right)\sqrt{\pi}}\int_0^\pi g(\phi_1)\sin^{k-2}(\phi_1)d\phi_1.
\end{equation}
We will now construct our first example. Given any probability density, $\varphi$, on $\mathbb{R}$ with $\text{Supp}(\varphi)\subset \left(0,\frac{1}{2} \right)$ we define $\varphi_\epsilon(x)=\frac{1}{\epsilon}\cdot \varphi
 \left(\frac{x}{\epsilon} \right)$ and $b_{\epsilon}(\phi)=\frac{\Gamma \left(\frac{N-1}{2} \right)\sqrt{\pi}}{\Gamma \left(\frac{N}{2} \right)}\cdot \frac{\varphi_{\epsilon}(\phi)}{\sin^{N-2}(\phi)}$. Let 
\begin{equation}\label{eq: 1st counter example}
F_N=\frac{1}{2^N}\sum_{i=1}^{2^N} b_{\epsilon_N}(\xi_i),
\end{equation}
where $\xi_i$ is the elevation angle with respect to a given $i-$th pole (i.e. $v_i=\pm \sqrt{N}$) and $\epsilon_N$ is a sequence converging to zero.
\begin{theorem}\label{thm: 1st counter example}
The family of probability densities $\left\lbrace F_N \right\rbrace_{N\in\mathbb{N}}$ defined in (\ref{eq: 1st counter example}) satisfies
\begin{equation}\label{eq: entropic horror I}
\lim_{N\rightarrow\infty}\frac{H_N(F_N)}{N}=\infty
\end{equation}
for any positive sequence $\left\lbrace\epsilon_N \right\rbrace_{N\in\mathbb{N}}$ that converges to zero.
\end{theorem}
\begin{proof}
Clearly $F_N$ is symmetric and due to (\ref{eq: integration on sphere with angle coordinates}) and its definition we find that $F_N$ is a probability density. Next we notice that due to symmetry and the fact that $b_{\epsilon_N}(\xi_i)$ are supported on disjoint sets we have that
\begin{equation}\label{eq: entropy of 1st counter example I}
\begin{gathered}
H_N(F_N)=\int_{\mathbb{S}^{N-1}\left( \sqrt{N}\right)}b_{\epsilon_N}(\xi_1)\log \left(\frac{\sum_{i=1}^N b_{\epsilon_N}(\xi_i)}{2^N} \right)d\sigma^N \\ = \int_{\mathbb{S}^{N-1}\left( \sqrt{N}\right)}b_{\epsilon_N}(\xi_1)\log \left(b_{\epsilon_N}(\xi_1) \right)d\sigma^N-N\log 2 \\
=\int_{0}^\pi \varphi_{\epsilon_N}(\xi)\log \left(\varphi_{\epsilon_N}(\xi) \right)d\xi + \log\left( \frac{\Gamma \left(\frac{N-1}{2} \right)\sqrt{\pi}}{\Gamma \left(\frac{N}{2} \right)} \right)\\
-(N-2)\int_{0}^\pi \varphi_{\epsilon_N}(\xi)\log \left(\sin(\xi)\right) d\xi -N\log2
\end{gathered}
\end{equation}

Using a change of variables $\xi=\frac{\xi}{\epsilon_N}$ and the fact that the support of $\varphi$ is in $\left(0,\frac{1}{2} \right)$, we find that for $N$ large enough
\begin{equation}\label{eq: dicomp entropy term I}
\begin{gathered}
\int_{0}^\pi \varphi_{\epsilon_N}(\xi)\log \left(\varphi_{\epsilon_N}(\xi) \right)d\xi
=\int_{0}^\pi \varphi(\xi)\log \left(\varphi(\xi) \right)d\xi-\log \epsilon_N,  
\end{gathered}
\end{equation}
as well as
\begin{equation}\label{eq: dicomp entropy term II}
\begin{gathered}
\int_{0}^\pi \varphi_{\epsilon_N}(\xi)\log \left(\sin(\xi)\right) d\xi
=\int_{0}^\pi \varphi(\xi)\log \left(\sin(\epsilon_N \xi)\right) d\xi \\
= \int_{0}^\pi \varphi(\xi)\log \left(\frac{\sin(\epsilon_N \xi)}{\epsilon_N \xi}\right) d\xi
+\log \epsilon_N +\int_{0}^\pi \varphi(\xi)\log \left(\xi\right) d\xi.
\end{gathered}
\end{equation}
When $N$ is large we find that $0<\frac{\sin(\epsilon_N \xi)}{\epsilon_N \xi}\leq 1$ and so (\ref{eq: dicomp entropy term II}) implies that
\begin{equation}\label{eq: dicomp entropy term III}
\begin{gathered}
\int_{0}^\pi \varphi_{\epsilon_N}(\xi)\log \left(\sin(\xi)\right) d\xi
\leq \log \epsilon_N +\int_{0}^\pi \varphi(\xi)\log \left(\xi\right) d\xi.
\end{gathered}
\end{equation}
Combining (\ref{eq: entropy of 1st counter example I}), (\ref{eq: dicomp entropy term I}), (\ref{eq: dicomp entropy term III}) and the approximation $\frac{\Gamma \left(\frac{N-1}{2} \right)\sqrt{\pi}}{\Gamma \left(\frac{N}{2}\right)}=\sqrt{\frac{2\pi}{N}}\left(1+O\left(\frac{1}{N} \right) \right)$ we find that 
\begin{equation}\label{eq: entropy of 1st counter example II}
\begin{gathered}
H_N(F_N)\geq \int_0^\pi \varphi(\xi)\log \left(\varphi(\xi)\right) d\xi+ \frac{\log\left(2\pi+O\left(\frac{1}{N} \right) \right)}{2}-\frac{\log N}{2}-N\log 2\\
 -(N-2)\int_0^\pi \varphi(\xi)\log \left(\xi\right) d\xi-(N-1)\log \epsilon_N.
\end{gathered}
\end{equation}
Thus
\begin{equation}\label{eq: 1st counter example finished}
\liminf_{N\rightarrow\infty}\frac{H_N(F_N)}{N}\geq \liminf_{N\rightarrow\infty}\left(-\log \epsilon_N\right) -\log 2 -\int_0^\pi \varphi(\xi)\log \left(\xi\right) d\xi
\end{equation}
proving the result.
\end{proof}

\section{Second Example: The Stereographic Projection.}\label{sec: stereographic}
Much like the previous section, we will once again construct a family of probability densities that satisfies $\lim_{N\rightarrow\infty}\frac{H_N(F_N)}{N}$. This time, however, we'd like to try and use $\mathbb{R}^{N-1}$ as our basis for construction and for that we will employ the stereographic projection.\\
Given a function $\zeta(x)$ on $\mathbb{R}^{N-1}$ we define its $i-$th extension to the sphere $\mathbb{S}^{N-1}(R)$ as
\begin{equation}\label{eq: ith extension to the speher}
J_{i,R}(v_1,\dots,v_N)=\frac{\left\lvert \mathbb{S}^{N-1} \right\rvert R^{2N-2}}{\left(R+v_i \right)^{N-1}}\zeta\circ S_i^{-1}(v_1,\dots,v_N),
\end{equation} 
where $S_i$ is the stereographic projection from $\mathbb{R}^{N-1}$ to $\mathbb{S}^{N-1}(R)$ with the $i-$th axis as the axis of symmetry. It is known that under $S_i$ we have
\begin{equation}\label{eq: norm of x^2}
|x|^2+R^2=\frac{2R^3}{R+v_i},
\end{equation}
and
\begin{equation}\label{eq: metric on the sphere}
ds_R=\left(\frac{2R^2}{R^2+|x|^2}\right)dx_1 \dots dx_{N-1}
\end{equation}
(see the Appendix for more information on the standard map with the $N-$th axis of symmetry). \\
We notice the following:
\begin{equation}\label{eq: connection between J and zeta integration}
\int_{\mathbb{S}^{N-1}(R)}J_{i,R}(v_1,\dots,v_N)d\sigma^N_R=\int_{\mathbb{S}^{N-1}(R)}\frac{R^{N-1}}{\left(R+v_i\right)^{N-1}}\cdot\zeta\circ S_i^{-1}(v_1,\dots,v_N)ds^N_R.
\end{equation}
Using (\ref{eq: norm of x^2}) and (\ref{eq: metric on the sphere}) we find that
\begin{equation}\label{eq: integral invariance under sphere extension}
\int_{\mathbb{S}^{N-1}(R)}J_{i,R}(v_1,\dots,v_N)d\sigma^N_R=\int_{\mathbb{R}^{N-1}}\zeta(x_1,\dots,x_{N-1})dx_1 \dots dx_{N-1}.
\end{equation}
Also, we find that
\begin{equation}\label{eq: connection between entropy of J and zeta I}
\begin{gathered}
\int_{\mathbb{S}^{N-1}(R)}J_{i,R}(v_1,\dots,v_N)\log \left(J_{i,R}(v_1,\dots,v_N) \right)d\sigma^N_R\\
=\int_{\mathbb{R}^{N-1}}\zeta(x_1,\dots,x_{N-1})\log \left(\zeta(x_1,\dots,x_{N-1}) \right)dx_1 \dots dx_{N-1}\\
+\int_{\mathbb{R}^{N-1}}\zeta(x_1,\dots,x_{N-1})\log \left(\frac{\left\lvert \mathbb{S}^{N-1} \right\rvert R^{2N-2}}{\left(R+v_i(x) \right)^{N-1}} \right)dx_1 \dots dx_{N-1},
\end{gathered}
\end{equation}
and applying (\ref{eq: norm of x^2} again shows that the last expression above equals to
\begin{equation}\label{eq: connection between entropy of J and zeta II}
\begin{gathered}
\left(\log\left(\left\lvert \mathbb{S}^{N-1} \right\rvert \right)-(N-1)\log(2R)\right)\int_{\mathbb{R}^{N-1}}\zeta(x_1,\dots,x_{N-1})dx_1 \dots dx_{N-1}\\
+(N-1)\int_{\mathbb{R}^{N-1}}\log\left(|x|^2+R^2 \right)\zeta(x_1,\dots,x_{N-1})dx_1 \dots dx_{N-1}.
\end{gathered}
\end{equation}
The approximation $\left\lvert \mathbb{S}^{N-1} \right\rvert=\left(\frac{2\pi}{e} \right)^{\frac{N}{2}}\cdot \frac{1+O\left(\frac{1}{N}\right)}{\sqrt{2\pi}N^{\frac{N-2}{2}}}$ helps us conclude that
\begin{equation}\label{eq: entropy change under sphere extension}
\begin{gathered}
\int_{\mathbb{S}^{N-1}(R)}J_{i,R}(v_1,\dots,v_N)\log \left(J_{i,R}(v_1,\dots,v_N) \right)d\sigma^N_R \\
=\int_{\mathbb{R}^{N-1}}\zeta(x_1,\dots,x_{N-1})\log \left(\zeta(x_1,\dots,x_{N-1}) \right)dx_1 \dots dx_{N-1}\\
+\left(\frac{N}{2}\cdot\log\left(\frac{2\pi}{e}\right)-\frac{N-2}{2}\cdot \log N -\frac{\log\left(2\pi\left(1+O\left(\frac{1}{N} \right) \right)\right)}{2} -(N-1)\log(2R)\right)\\ 
\cdot\int_{\mathbb{R}^{N-1}}\zeta(x_1,\dots,x_{N-1})dx_1 \dots dx_{N-1}+\\
(N-1)\int_{\mathbb{R}^{N-1}}\log\left(|x|^2+R^2 \right)\zeta(x_1,\dots,x_{N-1})dx_1 \dots dx_{N-1}.
\end{gathered}
\end{equation}
Lastly, in the case where $\zeta$ is a probability density on $\mathbb{R}^{N-1}$ (and thus $J_{i,R}$ by equation (\ref{eq: integral invariance under sphere extension})) we find that
\begin{equation}\label{eq: entropy of the extension}
\begin{gathered}
\frac{H_N(J_{i,\sqrt{N}})}{N}=\frac{\int_{\mathbb{R}^{N-1}}\zeta(x_1,\dots,x_{N-1})\log \left(\zeta(x_1,\dots,x_{N-1}) \right)dx_1 \dots dx_{N-1}}{N}\\
+\left(\frac{\log\left(\frac{2\pi}{e}\right)}{2}-\log N+\frac{3\log N}{2N} -\frac{\log\left(2\pi\left(1+O\left(\frac{1}{N} \right) \right)\right)}{2N} -\frac{(N-1)}{N}\log(2)\right)\\
\frac{(N-1)}{N}\cdot\int_{\mathbb{R}^{N-1}}\log\left(|x|^2+N \right)\zeta(x_1,\dots,x_{N-1})dx_1 \dots dx_{N-1}.
\end{gathered}
\end{equation}
The key observation here that all the integrals \emph{but the last one} are invariant under translation, and the last integration can be increased by shifting the bulk of $\zeta$ to infinity.\\
We are now ready to construct our second example: let $\zeta$ be any symmetric probability density on $\mathbb{R}$ that is supported on $[0,1]$. Define 
\begin{equation}\label{eq: definition of zeta_N}
\zeta_N(x_1,\dots,x_{N-1})=\prod_{i=1}^{N-1}\zeta(x_i-\beta_N),
\end{equation}
where $\beta_N$ will be chosen shortly, and
\begin{equation}\label{eq: stereographic example}
F_N(v_1,\dots,v_N)=\frac{\sum_{i=1}^N J_{i,N}(v_1,\dots,v_N)}{N},
\end{equation}
with $J_{i,N}$ defined by (\ref{eq: ith extension to the speher}) with $\zeta=\zeta_N$ and $R=\sqrt{N}$.
\begin{theorem}\label{thm: 2nd counter example}
The family of probability densities $\left\lbrace F_N \right\rbrace_{N\in\mathbb{N}}$ defined in (\ref{eq: stereographic example}) satisfies
\begin{equation}\label{eq: entropic horror II}
\lim_{N\rightarrow\infty}\frac{H_N(F_N)}{N}=\infty
\end{equation}
for any sequence $\left\lbrace\beta_N \right\rbrace_{N\in\mathbb{N}}$ such that $\lim_{N\rightarrow\infty}|\beta_N|=\infty$.
\end{theorem}

\begin{proof}
The first observation we make is since $\zeta_N$ is symmetric in its variables, $J_{i,N}$ is invariant under any change of variables that are not at the $i-$th position (see the Appendix for an explicit formula for $S_i$). Also, by the definition and the symmetry of $\zeta$, we have that
\[J_{i,N}\left(\dots,\underbrace{v_k}_{\text{$i-$th position}},\dots \right)=J_{k,N}\left(\dots,\underbrace{v_i}_{\text{$k-$th position}},\dots\right),\] 
and so, along with equation (\ref{eq: integral invariance under sphere extension}), we conclude that (\ref{eq: stereographic example}) is a symmetric probability density on $\mathbb{S}^{N-1}\left(\sqrt{N} \right)$.\\
The next observation we make is that
\begin{equation}\label{eq: zeta_N is entropic}
\begin{gathered}
\int_{\mathbb{R}^{N-1}}\zeta_N(x_1,\dots,x_{N-1})\log\left(\zeta_N(x_1,\dots,x_{N-1}) \right)dx_1\dots dx_{N-1}\\
=(N-1)\int_{\mathbb{R}}\zeta(x)\log\left(\zeta(x) \right)dx,
\end{gathered}
\end{equation}
and due to symmetry and monotonicity of the logarithm we have that
\begin{equation}\label{eq: evaluation of entropy}
\begin{gathered}
H_N(F_N)=\int_{\mathbb{S}^{N-1}\left(\sqrt{N} \right)} J_{1,N}(v_1,\dots,v_N)\log \left(\frac{\sum_{i=1}^N J_{i,N}(v_1,\dots,v_N)}{N} \right)d\sigma^N \\
\geq H_N(J_{1,N})-\log N.
\end{gathered}
\end{equation}
Combining (\ref{eq: entropy of the extension}), (\ref{eq: zeta_N is entropic}), (\ref{eq: evaluation of entropy}) along with the fact that if $x\in\text{supp}(\zeta_N)$ then $|x|^2\geq N\left(|\beta_N|-1 \right)^2$, we have that
\begin{equation}\label{eq: 2nd counter example finished}
\begin{gathered}
\liminf_{N\rightarrow\infty}\frac{H_N(F_N)}{N}\geq \int_{\mathbb{R}}\zeta(x)\log\left(\zeta(x) \right)dx
+\frac{\log(2\pi)-1}{2}-\log 2 \\
 +\liminf_{N\rightarrow\infty}\left(-\log N + \frac{N-1}{N}\cdot\log\left(N+N\left(|\beta_N|-1 \right)^2 \right)  \right) \\
 =\int_{\mathbb{R}}\zeta(x)\log\left(\zeta(x) \right)dx
+\frac{\log(2\pi)-1}{2}-\log 2 +\liminf_{N\rightarrow\infty}\left(\frac{N-1}{N}\cdot\log\left(1+\left(|\beta_N|-1 \right)^2 \right)  \right),
\end{gathered}
\end{equation}
proving the desired result. 
\end{proof}
The following sections will be of different flavour. We will no longer use the limiting convex combination idea but focus our attention on explicitly computable families of densities on the sphere. 
\section{Marginals of Densities on the Sphere}\label{sec: marginals on the sphere}
In this short section we will mention and prove some simple theorems about integration on the sphere, along with ways to identify marginals and chaoticity.\\
We start with an important Fubini-type formula, whose proof can be found in \cite{Einav1}:
\begin{lemma}\label{lem: fubini type formula on the sphere}
Let $F$ be a continuous function on $\mathbb{S}^{n-1}\left(r \right)$ then
\begin{equation}\label{eq: fubini on the sphere}
\begin{gathered}
\int_{\mathbb{S}^{n-1}\left( r \right)} F d\sigma^k_r=\frac{\left\lvert \mathbb{S}^{n-j-1} \right\rvert}{\left\lvert \mathbb{S}^{n-1} \right\rvert}\cdot \frac{1}{r^{n-2}}\cdot
\int_{\sum_{i=1}^j |v_i|^2 \leq r^2}\left(r^2-\sum_{i=1}^j |v_i|^2 \right)^{\frac{n-j-2}{2}}\\
\left(\int_{\mathbb{S}^{n-j-1}\left(\sqrt{r^2-\sum_{i=1}^j |v_i|^2} \right)}F d\sigma^{n-j}_{\sqrt{r^2-\sum_{i=1}^j |v_i|^2}}\right)dv_1\dots dv_j.
\end{gathered}
\end{equation} 
\end{lemma}
An immediate corollary is the following:
\begin{corollary}\label{cor: marginals}
Let $F_N$ be continuous on $\mathbb{S}^{N-1}\left(\sqrt{N} \right)$ then 
\begin{equation}\label{eq: kth marginal}
\begin{gathered}
\Pi_k\left( F_N \right)(v_1,\dots,v_k)=\frac{\left\lvert \mathbb{S}^{N-k-1} \right\rvert}{\left\lvert \mathbb{S}^{N-1} \right\rvert}\cdot \frac{\left(N-\sum_{i=1}^k |v_i|^2 \right)_{+}^{\frac{N-k-2}{2}}}{N^\frac{N-2}{2}}\\
\left(\int_{\mathbb{S}^{N-k-1}\left(\sqrt{N-\sum_{i=1}^k |v_i|^2} \right)}F_N d\sigma^{n-k}_{\sqrt{N-\sum_{i=1}^k |v_i|^2}}\right),
\end{gathered}
\end{equation}
where $f_{+}=\max(f,0)$.
\end{corollary}
Next, we prove a simple technical lemma that will be very useful in determining when a family of probability densities is chaotic.
\begin{lemma}\label{lem: easy weak convergence}
Let $\left\lbrace f_n \right\rbrace_{N\in\mathbb{N}}$ be a sequence of non-negative function on $\mathbb{R}^k$ that converges pointwise to a function $f\in L^1\left(\mathbb{R}^k \right)$. If in addition, 
\[\lim_{n\rightarrow\infty}\int_{\mathbb{R}^k}f_n(x_1,\dots,x_k)dx_1\dots dx_k=\int_{\mathbb{R}^k}f(x_1,\dots,x_k)dx_1\dots dx_k,\]
then $f_n\in L^1\left(\mathbb{R}^k \right)$ from a certain $n_0\in\mathbb{N}$, and $\left\lbrace f_n \right\rbrace_{N\in\mathbb{N}}$ converges to $f$ in $L^1\left(\mathbb{R}^k \right)$. 
\end{lemma}
\begin{proof}
It is easy to see that due to the conditions of the Lemma we have that $f$ is non-negative and that $f_n \in L^1\left(\mathbb{R}^k \right)$ from a certain $n_0$. Without loss of generality we can assume that $n_0=1$. Define 
\begin{equation}\label{eq: def g_n and g}
g_n=f_n+f, \quad g=2f.
\end{equation}
Clearly $g,g_n \geq 0$, $g,g_n \in L^1 \left(\mathbb{R}^k \right)$, $g_n$ converges to $g$ pointwise and 
\begin{equation}\label{eq: int of g_n converges to int of g}
\begin{gathered}
\lim_{n\rightarrow\infty}\int_{\mathbb{R}^k}g_n(x_1,\dots,x_k)dx_1\dots dx_k=\int_{\mathbb{R}^k}g(x_1,\dots,x_k)dx_1\dots dx_k.
\end{gathered}
\end{equation}
Since $\left\lvert f_n-f\right\rvert \leq g_n$ and $\left\lvert f_n-f\right\rvert$ converges pointwise to zero, we conclude the desired result from Lebesgue's generalised dominated convergence theorem. 
\end{proof}

From the above lemma we can deduce the following:
\begin{corollary}\label{cor: marginal and chaoticity}
Let $\left\lbrace F_N \right\rbrace_{N\in\mathbb{N}}$ be a sequence of probability densities on Kac's sphere. If there exists a probability density function, $f$, on $\mathbb{R}$ such that
\[\lim_{N\rightarrow\infty}\Pi_k \left(F_N \right)=f^{\otimes k}\]
pointwise for all $k\in\mathbb{N}$, then $F_N$ is $f-$chaotic.
\end{corollary}
\begin{proof}
This follows immediately from Lemma \ref{lem: easy weak convergence} and the fact that
\begin{equation}\label{eq: int convergence of marginals to f^otimes k}
\begin{gathered}
\int_{\mathbb{R}^k}\Pi_k \left(F_N \right)(v_1,\dots,v_k)dv_1 \dots dv_k=\int_{\mathbb{S}^{N-1}\left(\sqrt{N} \right)}F_N d\sigma^N\\
=1=\int_{\mathbb{R}^k}f^{\otimes k}(v_1,\dots,v_k)dv_1 \dots dv_k
\end{gathered}
\end{equation}
For all $k,N\in\mathbb{N}$.
\end{proof}
Armed with our new tools, we are now ready to give two more examples of chaotic families that are not entropically chaotic.

\section{Third Example: An Escaping Tensorisation.}\label{sec: escaping tensorisation}
The third example we'll construct has the intuitive form of a tensorised product restricted to the sphere with one major difference: The underlying one dimensional function depends on $N$ in such a way that the family will lose part of its support at infinity, ruining the entropic chaoticity. Most of the computations presented in this section are taken from the author's previous work \cite{Einav1}, but a few will be repeated for the sake of completion.\\
Our family of interest is defined by (\ref{eq: usual suspect}) where $f_N(v)=\delta_N M_{\frac{1}{2\delta_N}}(v)+(1-\delta_N) M_{\frac{1}{2(1-\delta_N)}}(v)$ with $M_a(v)=\frac{e^{-\frac{v^2}{2a}}}{\sqrt{2 \pi a}}$ and $\delta_N=\frac{1}{N^{\eta}}$, $\eta$ close to $1$. Defining the normalization function as
\begin{equation}\label{eq: normalization function def}
\mathcal{Z}_N(f_N,\sqrt{r})=\int_{\mathbb{S}^{N-1}(r)}\prod_{i=1}^N f_N(v_i)d\sigma^N_r,
\end{equation}
we see that
\begin{equation}\label{eq: 3rd counter example def}
F_N(v_1,\dots,v_N)=\frac{\prod_{i=1}^N f_N(v_i)}{\mathcal{Z}_N(f_N,\sqrt{N})}.
\end{equation}
The goal of this section is to prove Theorem \ref{thm: entropic fail on usual suspect}, showing that $\left\lbrace F_N \right\rbrace_{N\in\mathbb{N}}$ is chaotic, but not entropically chaotic. In order to do that we require a few additional computations and technical lemmas, first amongst them is an explicit asymptotic expression to the normalization function $\mathcal{Z}_N$. This part is quite lengthy and technical and is fully proved in \cite{Einav1}. As such, we will content ourselves with stating the final result:
\begin{lemma}\label{lem: approximation lemma}
Let $\mathcal{Z}_N$ defined as in (\ref{eq: normalization function def}), then
\begin{equation}\label{eq: normalization approximation}
\begin{gathered}
\mathcal{Z}_N(f_N,\sqrt{u})=\frac{2}{\sqrt{N}\Sigma_N \left\lvert \mathbb{S}^{N-1} \right\rvert|u|^{\frac{N-2}{2}}}\left(\frac{e^{-\frac{(u-N)^2}{2N\Sigma^2_N}}}{\sqrt{2\pi}}+\lambda_N(u) \right),
\end{gathered}
\end{equation}
where $\Sigma_N^2=\frac{3}{4\delta_N(1-\delta_N)}-1$ and $\lim_{N\rightarrow\infty}\left(\sup_u \left\lvert \lambda_N(u) \right\rvert\right)=0$.
\end{lemma}
Using this approximation we can now discuss the chaoticity of $F_N$.

\begin{lemma}\label{lem: tensorised F_N is chaotic}
The family of probability densities, $\left\lbrace F_N \right\rbrace_{N\in\mathbb{N}}$ is $M_{\frac{1}{2}}-$chaotic.
\end{lemma}
\begin{proof}
Using Corollary \ref{cor: marginals} and the definition of the normalization function we find that
\begin{equation}\label{eq: kth marginal of the tensorisation I}
\begin{gathered}
\Pi_k\left( F_N \right)(v_1,\dots,v_k)=\frac{\left\lvert \mathbb{S}^{N-k-1} \right\rvert}{\left\lvert \mathbb{S}^{N-1} \right\rvert}\cdot \frac{\left(N-\sum_{i=1}^k |v_i|^2 \right)_{+}^{\frac{N-k-2}{2}}}{N^\frac{N-2}{2}}\\
\cdot\frac{\mathcal{Z}_{N-k}\left(f_N,\sqrt{N-\sum_{i=1}^k |v_i|^2} \right)}{\mathcal{Z}_N(f_N,\sqrt{N})}\cdot \left(\prod_{I=1}^k f_N(v_i)\right).
\end{gathered}
\end{equation}
Combining this with (\ref{eq: normalization approximation}) yields
\begin{equation}\label{eq: kth marginal of the tensorisation II}
\begin{gathered}
\Pi_k\left( F_N \right)(v_1,\dots,v_k)=\sqrt{\frac{N}{N-k}}\cdot \frac{e^{-\frac{\left(k-\sum_{i=1}^k |v_i|^2 \right)^2}{2(N-k)\Sigma_N^2}}+\lambda_{N-k}\left(N-\sum_{i=1}^k |v_i|^2 \right)}{1+\lambda_N(N)}\\
\left(\prod_{I=1}^k f_N(v_i)\right)\chi_{\sum_{i=1}^k |v_i|^2 \leq N}(v_1,\dots,v_k).
\end{gathered}
\end{equation}
From Lemma \ref{lem: approximation lemma} we see that $\lim_{N\rightarrow\infty}\left(\sup |\lambda_{N-j}| \right)=0$ for any fixed $j$, and by its definition and our choice of $\delta_N$ we have that $\lim_{N\rightarrow\infty}\Sigma_N^2=\infty$. We conclude that
\begin{equation}\label{eq: pointwise limit of kth marginal}
\lim_{N\rightarrow\infty}\Pi_k \left(F_N \right)(v_1,\dots,v_k)=M_{\frac{1}{2}}^{\otimes k}(v_1,\dots,v_k)
\end{equation}
pointwise, as $f_N$ clearly converges to $M_{\frac{1}{2}}$ pointwise. This is enough to prove the desired result due to Corollary \ref{cor: marginal and chaoticity}.
\end{proof}
Next, we compute the rescaled $N-$particle entropy of $F_N$.
\begin{lemma}\label{lem: computation of the entropy}
\begin{equation}
\lim_{N\rightarrow\infty}\frac{H_N(F_N)}{N}=\frac{\log 2}{2}.
\end{equation}
\end{lemma}
We will give a quick sketch of the proof, and direct the reader to \cite{Einav1} for full details.
\begin{proof}
Due to symmetry, Lemma \ref{lem: fubini type formula on the sphere} and Lemma \ref{lem: approximation lemma} we have that
\begin{equation}\label{eq: entropic estimation of tensorisation I}
\begin{gathered}
\frac{H_N(F_N)}{N}=\frac{\left\lvert \mathbb{S}^{N-2} \right\rvert}{\left\lvert \mathbb{S}^{N-1} \right\rvert}
\cdot \int_{-\sqrt{N}}^{\sqrt{N}} \frac{\left(N-v_1^2 \right)^{\frac{N-3}{2}}}{N^{\frac{N-2}{2}}}
\cdot \frac{\mathcal{Z}_{N-1}\left(f_N,\sqrt{N-v_1^2}\right)}{\mathcal{Z}_N(f_N,\sqrt{N}}\\
\cdot f_N(v_1)\log\left(f_N(v_1) \right)dv_1 - \frac{\log\left(\mathcal{Z}_N(f_N,\sqrt{N})\right)}{N}\\
=\sqrt{\frac{N}{N-1}}\int_{-\sqrt{N}}^{\sqrt{N}} \frac{e^{-\frac{\left(1-v_1^2\right)^2}{2(N-1)\Sigma_N^2}}+\lambda_{N-1}\left(N-|v_1|^2 \right)}{1+\lambda_N(N)}\cdot f_N(v_1)\log\left(f_N(v_1) \right)dv_1\\
+\frac{\log\left(\sqrt{2\pi}\Sigma_N \left\lvert \mathbb{S}^{N-1} \right\rvert N^{\frac{N}{2}} \right)-\log \left(2(1+\lambda_N(N)) \right)}{N}.
\end{gathered}
\end{equation}
Using the Generalised Dominated Convergence theorem one can show that 
\begin{equation}\label{eq: entropic estimation of tensorisation II}
\begin{gathered}
\int_{-\sqrt{N}}^{\sqrt{N}} \frac{e^{-\frac{\left(1-v_1^2\right)^2}{2(N-1)\Sigma_N^2}}+\lambda_{N-1}\left(N-|v_1|^2 \right)}{1+\lambda_N(N)}\cdot f_N(v_1)\log\left(f_N(v_1) \right)dv_1\\
\underset{N\rightarrow\infty}{\longrightarrow}\int_{\mathbb{R}}M_{\frac{1}{2}}(v_1)\log \left( M_{\frac{1}{2}}(v_1) \right)dv_1.
\end{gathered}
\end{equation}
That, along with approximation for $\left\lvert \mathbb{S}^{N-1} \right\rvert$, gives the desired result.
\end{proof}
\begin{proof}[Proof of Theorem \ref{thm: entropic fail on usual suspect}]
From Lemma \ref{lem: tensorised F_N is chaotic}  We know that $F_N$ is $M_{\frac{1}{2}}-$chaotic and from Lemma \ref{lem: computation of the entropy} we know that $\lim_{N\rightarrow\infty}\frac{H_N(F_N)}{N}=\frac{\log 2}{2}$. However,
\begin{equation}\label{eq: M_a log M_a}
\int_{\mathbb{R}}M_a(v)\log\left(M_a(v) \right)dv=-\frac{\log(2\pi a)}{2}-\frac{1}{2},
\end{equation}
and
\begin{equation}\label{eq: M_a log gamma}
\int_{\mathbb{R}}M_a(v)\log\left(\gamma(v) \right)dv=-\frac{\log(2\pi)}{2}-\frac{a}{2}.
\end{equation}
Thus,
\begin{equation}\label{eq: actual entropy of limit}
H \left( M_\frac{1}{2} | \gamma \right)=\frac{\log 2}{2}-\frac{1}{4}<\frac{\log 2}{2},
\end{equation}
concluding the proof.
\end{proof}

\section{Fourth Example: Varying Polynomials.}\label{sec: polynomials}
The last example we will provide in this paper is a family of probability densities on the sphere that is made of symmetric polynomial with varying degrees, constrained to the sphere. Surprisingly enough, we can compute the normalization function very easily in this case and we will see that the reason for this example's failure to be entropically chaotic is its 'large' entropic tails.\\
In order to emphasize the effect of varying powers in our subsequent paragraphs we will define two families of probability densities, both of similar 'flavour' but very different properties (one was mentioned in Theorem \ref{thm: entropci fail on polynomial}). \\
Let
\begin{equation}\label{eq: def of f_N,m}
f_{N,m}(v_1,\dots,v_N)= \sum_{i=1}^N |v_i|^m,
\end{equation} 
where $m>0$. Denote by $f_N=f_{N,N}$ and let $\mathfrak{Z}_{N,m},\mathfrak{Z}_N$ be the appropriate normalization functions on Kac's sphere.\\
Our main two families of interest are:
\begin{equation}\label{eq: def of 4th counter example}
\begin{gathered}
F_{N,m}(v_1,\dots,v_N)=\frac{f_{N,m}(v_1,\dots,v_N)}{\mathfrak{Z}_{N,m}}, \\
F_{N}(v_1,\dots,v_N)=\frac{f_{N}(v_1,\dots,v_N)}{\mathfrak{Z}_{N}},
\end{gathered}
\end{equation}
where $m$ is fixed in the first family. The main result of this section is the following:
\begin{theorem}\label{thm: 4th counter example}
The family of probability densities $\left\lbrace F_{N,m} \right\rbrace_{N\in\mathbb{N}}$, defined in (\ref{eq: def of 4th counter example}), is $\gamma-$entropically chaotic while the family $\left\lbrace F_{N} \right\rbrace_{N\in\mathbb{N}}$ is $M_{\frac{1}{2}}-$chaotic, but not entropically chaotic.
\end{theorem}
which will also prove Theorem \ref{thm: entropci fail on polynomial}.The proof of this theorem will involve a few steps. We start with a few computations. 
\begin{lemma}\label{lem: integration of |v|^m}
Let $m>-1$. Then 
\begin{equation}\label{eq: integration of |v|^m}
\int_{\mathbb{S}^{N-1}(r)}|v_1|^m d\sigma^N_r = \frac{r^m \cdot\Gamma\left(\frac{N}{2} \right)\cdot\Gamma\left(\frac{m+1}{2} \right)}{\sqrt{\pi}\cdot\Gamma\left(\frac{N+m}{2} \right)}.
\end{equation}
\end{lemma} 
\begin{proof}
Using Lemma \ref{lem: fubini type formula on the sphere} we find that
\begin{equation}\label{eq: computation of |v|^m I}
\begin{gathered}
\int_{\mathbb{S}^{N-1}(r)}|v_1|^m d\sigma^N_r=\frac{\left\lvert \mathbb{S}^{N-2} \right\rvert}{\left\lvert \mathbb{S}^{N-1} \right\rvert}\cdot \frac{1}{r^{N-2}}\int_{-r}^{r}|v_1|^m\left(r^2-v_1^2 \right)^{\frac{N-3}{2}}dv_1 \\
=\frac{2 r^m \Gamma\left(\frac{N}{2} \right)}{\sqrt{\pi}\Gamma\left(\frac{N-1}{2} \right)}
\int_{0}^{1}x^m\left(1-x^2 \right)^{\frac{N-3}{2}}dx,
\end{gathered}
\end{equation}
where we used the substitution $v_1=rx$ and the formula 
\begin{equation}\label{eq: connection between S^N-1 and Gamma}
\left\lvert \mathbb{S}^{N-1} \right\rvert=\frac{2\pi^{\frac{N}{2}}}{\Gamma\left(\frac{N}{2} \right)}.
\end{equation} 
Equation (\ref{eq: Beta formula}) as well the identity
\begin{equation}\label{eq: connection between beta and gamma}
B(x,y)=\frac{\Gamma(x)\Gamma(y)}{\Gamma(x+y)},
\end{equation}
simplify (\ref{eq: computation of |v|^m I}) to the desired result.
\end{proof}
\begin{corollary}\label{cor: expression for Z_N,m and Z_N}
\begin{equation}\label{eq: expression for Z_N,m}
\begin{gathered}
\mathfrak{Z}_{N,m}=\frac{N\cdot 2^{\frac{m}{2}}\cdot \Gamma \left(\frac{m+1}{2} \right)}{\sqrt{\pi}}\cdot(1+\epsilon_N),
\end{gathered}
\end{equation}
\begin{equation}\label{eq: expression for Z_N}
\begin{gathered}
\mathfrak{Z}_{N}=\frac{N^{\frac{N+2}{2}}}{2^{N-1}},
\end{gathered}
\end{equation}
where $\epsilon_N$ goes to zero as $N$ goes to infinity.
\end{corollary}
\begin{proof}
We start by noticing that due to symmetry and Lemma \ref{lem: integration of |v|^m} we have that
\begin{equation}\label{eq: Z_N,m computation I}
\begin{gathered}
\mathfrak{Z}_{N,m}=N\cdot \int_{\mathbb{S}^{N-1}\left(\sqrt{N}\right)}|v_1|^m d\sigma^N 
=\frac{N^{\frac{m+2}{2}} \cdot\Gamma\left(\frac{N}{2} \right)\cdot\Gamma\left(\frac{m+1}{2} \right)}{\sqrt{\pi}\cdot\Gamma\left(\frac{N+m}{2} \right)}.
\end{gathered}
\end{equation}
Next, we see that the approximation
\begin{equation}\label{eq: Gamma approximation}
\Gamma(z)=z^{z-\frac{1}{2}}\cdot e^{-z}\cdot \sqrt{2\pi}\left(1+\frac{1}{12z}+\dots \right),
\end{equation}
for large $z$, leads to
\begin{equation}\label{eq: Gamma / Gamma approximation}
\begin{gathered}
\frac{\Gamma \left(\frac{N}{2} \right)}{\Gamma \left(\frac{N+m}{2} \right)}=\frac{1+\epsilon_N}{\left( \frac{N}{2} \right)^{\frac{m}{2}}},
\end{gathered}
\end{equation}
where $\epsilon_N$ goes to zero as $N$ goes to infinity.
Combining (\ref{eq: Z_N,m computation I}) and (\ref{eq: Gamma / Gamma approximation}) yields (\ref{eq: expression for Z_N,m}).\\
Similarly, by plugging $m=N$ in (\ref{eq: integration of |v|^m}) we find that
\begin{equation}\label{eq: Z_N computation I}
\begin{gathered}
\mathfrak{Z}_{N}=N\cdot \int_{\mathbb{S}^{N-1}\left(\sqrt{N}\right)}|v_1|^N d\sigma^N 
=\frac{N^{\frac{N+2}{2}} \cdot\Gamma\left(\frac{N}{2} \right)\cdot\Gamma\left(\frac{N+1}{2} \right)}{\sqrt{\pi}\cdot\Gamma\left(N \right)}.
\end{gathered}
\end{equation}
The known formula
\begin{equation}\label{eq: Gamma times Gamma}
\Gamma \left(z \right)\cdot \Gamma \left(z+\frac{1}{2} \right)=2^{1-2z}\cdot \sqrt{\pi}\cdot \Gamma \left(2z \right),
\end{equation}
together with (\ref{eq: Z_N computation I}) yields (\ref{eq: expression for Z_N}). 
\end{proof}

We are now ready to start proving Theorem \ref{thm: 4th counter example}.
\begin{lemma}\label{lem: F_N,m is entropically chaotic}
The family of probability densities $\left\lbrace F_{N,m} \right\rbrace_{N\in\mathbb{N}}$ is  $\gamma-$entropically chaotic.
\end{lemma}
\begin{proof}
In \cite{CCRLV} the authors showed that if $\left\lbrace F_N \right\rbrace_{N\in\mathbb{N}}$ is a symmetric family of probability densities such that $\lim_{N\rightarrow\infty}\frac{H_N(F_N)}{N}=0$ then the family is $\gamma-$entropically chaotic (they have actually proved something stronger than that). Thus, we only need to show that
\begin{equation}\label{eq: entropci chaoticity of F_N,m}
\lim_{N\rightarrow\infty}\frac{H_N(F_{N,m})}{N}=0.
\end{equation}
Indeed, from (\ref{eq: expression for Z_N,m}) we see that $\lim_{N\rightarrow\infty}\frac{\log\left(\mathfrak{Z}_{N,m} \right)}{N}=0$ and since on Kac's sphere $f_{N,m}(v_1,\dots,v_k) \leq N^{\frac{m}{2}}$ we find that
\begin{equation}\label{eq: entropy estimation of F_N,m}
\begin{gathered}
0\leq H_N(F_{N,m})\\
=\frac{1}{\mathfrak{Z}_{N,m}}\int_{\mathbb{S}^{N-1}\left(\sqrt{N} \right)}f_{N,m}(v_1,\dots,v_k)\log\left(f_{N,m}(v_1,\dots,v_k) \right)d\sigma^N  - \log \mathfrak{Z}_{N,m} \\
\leq \frac{m\log N}{2}-\log \mathfrak{Z}_{N,m},
\end{gathered}
\end{equation}
which shows (\ref{eq: entropci chaoticity of F_N,m}).
\end{proof}

We now turn our attention to the family $\left\lbrace F_N \right\rbrace_{N\in\mathbb{N}}$.
\begin{lemma}\label{lem: F_N is chaotic}
The family of probability densities $\left\lbrace F_{N} \right\rbrace_{N\in\mathbb{N}}$ is $M_{\frac{1}{2}}-$chaotic.
\end{lemma}
\begin{proof}
We start with Corollary \ref{cor: marginals} and the $k-$th marginal:
\begin{equation}\label{eq: kth marginal of F_N I}
\begin{gathered}
\Pi_k \left(F_{N} \right)(v_1,\dots,v_k)=\frac{\left\lvert \mathbb{S}^{N-k-1} \right\rvert}{\left\lvert \mathbb{S}^{N-1} \right\rvert} \cdot \frac{\left(N-\sum_{i=1}^k |v_i|^2 \right)_{+}^{\frac{N-k-2}{2}}}{N^{\frac{N-2}{2}}\cdot \mathfrak{Z}_{N}}\\
\cdot \left(\sum_{i=1}^k |v_i|^N +(N-k)\int_{\mathbb{S}^{N-k-1}\left(\sqrt{N-\sum_{i=1}^k |v_i|^2} \right)}|v_{k+1}|^N d\sigma^{N-k}_{\sqrt{N-\sum_{i=1}^N |v_i|^2}}\right).
\end{gathered}
\end{equation}
Next, we use Lemma \ref{lem: integration of |v|^m} to find that
\begin{equation}\label{eq: F_N is chaotic computation I}
\begin{gathered}
\int_{\mathbb{S}^{N-k-1}\left(\sqrt{N-\sum_{i=1}^k |v_i|^2} \right)}|v_{k+1}|^N d\sigma^{N-k}_{\sqrt{N-\sum_{i=1}^k |v_i|^2}}\\
=\frac{N^{\frac{N}{2}} \cdot\Gamma\left(\frac{N-k}{2} \right)\cdot\Gamma\left(\frac{N+1}{2} \right)}{\sqrt{\pi}\cdot\Gamma\left(N-\frac{k}{2} \right)}\left(1-\frac{\sum_{i=1}^k |v_i|^2}{N} \right)^{\frac{N}{2}}.
\end{gathered}
\end{equation}
From expression (\ref{eq: Gamma approximation}) we see that
\begin{equation}\label{eq: lots of Gammas}
\begin{gathered}
\Gamma \left(\frac{N-k}{2} \right)=\left(\frac{N}{2} \right)^{\frac{N-k-1}{2}}\left(1-\frac{k}{N} \right)^{\frac{N-k-1}{2}}\cdot e^{\frac{-N+k}{2}}\cdot \sqrt{2\pi}(1+\epsilon_N),\\
\Gamma \left(\frac{N+1}{2} \right)=\left(\frac{N}{2} \right)^{\frac{N}{2}}\left(1+\frac{1}{N} \right)^{\frac{N}{2}}\cdot e^{\frac{-N-1}{2}}\cdot \sqrt{2\pi}(1+\epsilon_N),\\
\Gamma \left(N-\frac{k}{2} \right)=N^{N-\frac{k+1}{2}}\cdot\left(1-\frac{k}{2N} \right)^{N-\frac{k+1}{2}}\cdot e^{-N+\frac{k}{2}}\cdot \sqrt{2\pi}(1+\epsilon_N),
\end{gathered}
\end{equation}
leading to
\begin{equation}\label{eq: F_N is chaotic computation II}
\begin{gathered}
\int_{\mathbb{S}^{N-k-1}\left(\sqrt{N-\sum_{i=1}^k |v_i|^2} \right)}|v_{k+1}|^N d\sigma^{N-k}_{\sqrt{N-\sum_{i=1}^k |v_i|^2}}\\
=\frac{N^{\frac{N}{2}} \cdot 2^{\frac{k}{2}}}{2^{N-1}}\left(1-\frac{\sum_{i=1}^k |v_i|^2}{N} \right)^{\frac{N}{2}}(1+\epsilon_N).
\end{gathered}
\end{equation}
Combining (\ref{eq: F_N is chaotic computation I}), , (\ref{eq: connection between S^N-1 and Gamma}), (\ref{eq: Gamma / Gamma approximation}), (\ref{eq: expression for Z_N}) and (\ref{eq: F_N is chaotic computation II}) we find that
\begin{equation}\label{eq: kth marginal of F_N II}
\begin{gathered}
\Pi_k \left(F_N,m \right)(v_1,\dots,v_k)=\frac{N^{\frac{k}{2}}}{\pi^{\frac{k}{2}}\cdot 2^{\frac{k}{2}}} \cdot \frac{2^{N-1}\left(N-\sum_{i=1}^k |v_i|^2 \right)_{+}^{\frac{N-k-2}{2}}}{N^{\frac{N-2}{2}}\cdot N^{\frac{N+2}{2}} }\\
\cdot \left(\sum_{i=1}^k |v_i|^N +\frac{\left(1-\frac{k}{N} \right)\cdot N^{\frac{N+2}{2}}\cdot 2^{\frac{k}{2}}}{2^{N-1}}\cdot\left(1-\frac{\sum_{i=1}^k |v_i|^2}{N} \right)^{\frac{N}{2}}\cdot(1+\epsilon_N)\right)\\
=\left( 1-\frac{\sum_{i=1}^k |v_i|^2}{N} \right)_{+}^{\frac{N-k-2}{2}}\left( \frac{2^{N-1}\cdot\sum_{i=1}^k |v_i|^N}{\left(2\pi\right)^{\frac{k}{2}}\cdot N^{\frac{N+2}{2}}}+\left(1-\frac{k}{N}\right)\cdot \frac{\left(1-\frac{\sum_{i=1}^k |v_i|^2}{N} \right)^{\frac{N}{2}}}{\pi ^{\frac{k}{2}}} \right)\\
\cdot (1+\epsilon_N).
\end{gathered}
\end{equation}
Clearly, we have that
\begin{equation}\label{eq: kth marginal of F_N III}
\begin{gathered}
\Pi_k \left(F_N \right)(v_1,\dots,v_k)\underset{N\rightarrow\infty}{\longrightarrow}M_{\frac{1}{2}}^{\otimes k}(v_1,\dots,v_k)
\end{gathered}
\end{equation}
pointwise, which finishes the proof due to Corollary \ref{cor: marginal and chaoticity}.
\end{proof}
Before we show the final stage in the proof of Theorem \ref{thm: 4th counter example} we require the following technical lemma:
\begin{lemma}\label{lem: high entropic tails}
\begin{equation}\label{eq: high entropic tails}
\int_{\mathbb{S}^{N-1}\left( \sqrt{N}\right)}|v_1|^N \log \left( |v_1|^N \right)d\sigma^N
\geq \frac{\mathfrak{Z}_N\cdot \log N}{2}-\frac{\mathfrak{Z}_N \cdot \log 2}{2}\cdot(1+\epsilon_N),
\end{equation}
where $\epsilon_N$ goes to zero as $N$ goes to infinity.
\end{lemma}
\begin{proof}
Using equation (\ref{eq: fubini on the sphere}) we see that
\begin{equation}\label{eq: entropic tails computation I}
\begin{gathered}
\int_{\mathbb{S}^{N-1}\left( \sqrt{N}\right)}|v_1|^N \log \left( |v_1|^N \right)d\sigma^N
=N\cdot \frac{\left\lvert \mathbb{S}^{N-2} \right\rvert}{\left\lvert \mathbb{S}^{N-1} \right\rvert}\\
\cdot \frac{1}{N^{\frac{N-2}{2}}}\int_{-\sqrt{N}}^{\sqrt{N}}|v_1|^N \left(N-v_1^2 \right)^{\frac{N-3}{2}}\log|v_1| dv_1= N\cdot \frac{\left\lvert \mathbb{S}^{N-2} \right\rvert}{\left\lvert \mathbb{S}^{N-1} \right\rvert}\\
\cdot N^{\frac{N}{2}}\int_{-1}^{1}|x|^N \left(1-x^2 \right)^{\frac{N-3}{2}}\log\left(\sqrt{N}|x| \right)dx,
\end{gathered}
\end{equation}
where we used the change of variables $v_1=\sqrt{N}x$. Similarly one can show that
\begin{equation}\label{eq: entropic tails computation II}
\begin{gathered}
\frac{\mathfrak{Z}_N}{N}=\int_{\mathbb{S}^{N-1}\left( \sqrt{N}\right)}|v_1|^Nd\sigma^N
=\frac{\left\lvert \mathbb{S}^{N-2} \right\rvert}{\left\lvert \mathbb{S}^{N-1} \right\rvert}\cdot N^{\frac{N}{2}}
\int_{-1}^{1}|x|^N \left(1-x^2 \right)^{\frac{N-3}{2}}dx,
\end{gathered}
\end{equation}
and thus
\begin{equation}\label{eq: entropic tails computation III}
\begin{gathered}
\int_{\mathbb{S}^{N-1}\left( \sqrt{N}\right)}|v_1|^N \log \left( |v_1|^N \right)d\sigma^N
=\frac{\mathfrak{Z}_N\cdot \log N}{2}\\
+N\cdot \frac{\left\lvert \mathbb{S}^{N-2} \right\rvert}{\left\lvert \mathbb{S}^{N-1} \right\rvert}\cdot N^{\frac{N}{2}}\int_{-1}^{1}|x|^N \left(1-x^2 \right)^{\frac{N-3}{2}}\log|x|dx.
\end{gathered}
\end{equation}
Using the simple inequality 
\begin{equation}\label{eq: logarithm ineq}
t^{\alpha}\log t \geq -\frac{1}{\alpha\cdot e },
\end{equation}
for $t>0$ and fixed $\alpha>0$, we find that
\begin{equation}\label{eq: entropic tails computation IV}
\begin{gathered}
\frac{\left\lvert \mathbb{S}^{N-2} \right\rvert}{\left\lvert \mathbb{S}^{N-1} \right\rvert}\int_{-1}^{1}|x|^N \left(1-x^2 \right)^{\frac{N-3}{2}}\log|x|dx\\
\geq -\frac{\Gamma\left(\frac{N}{2} \right)}{\sqrt{\pi}\Gamma\left(\frac{N-1}{2} \right)}\cdot\frac{1}{\alpha \cdot e}\int_{-1}^{1}|x|^{N-\alpha} \left(1-x^2 \right)^{\frac{N-3}{2}}dx\\
=-\frac{\Gamma\left(\frac{N}{2} \right)}{\sqrt{\pi}\Gamma\left(\frac{N-1}{2} \right)}\cdot\frac{B\left(\frac{N-\alpha+1}{2}, \frac{N-1}{2} \right)}{\alpha \cdot e}
=-\frac{\Gamma\left(\frac{N-\alpha+1}{2} \right)\Gamma\left(\frac{N}{2} \right)}{\sqrt{\pi}\cdot\alpha \cdot e\cdot \Gamma\left(N-\frac{\alpha}{2} \right)}
\end{gathered}
\end{equation}
Similar to equations (\ref{eq: lots of Gammas}) we can easily show that
\begin{equation}\label{eq: entropic tails computation V}
\begin{gathered}
\frac{\left\lvert \mathbb{S}^{N-2} \right\rvert}{\left\lvert \mathbb{S}^{N-1} \right\rvert}\int_{-1}^{1}|x|^N \left(1-x^2 \right)^{\frac{N-3}{2}}\log|x|dx \geq -\frac{2^{\frac{\alpha}{2}}(1+\epsilon_N)}{2^{N-1}\cdot\alpha \cdot e }.
\end{gathered}
\end{equation}
Chosen to optimize (\ref{eq: entropic tails computation V}) we pick $\alpha=\frac{2}{\log 2}$ and conclude that
\begin{equation}\label{eq: entropic tails computation VI}
\begin{gathered}
\frac{\left\lvert \mathbb{S}^{N-2} \right\rvert}{\left\lvert \mathbb{S}^{N-1} \right\rvert}\cdot N^{\frac{N+2}{2}}\int_{-1}^{1}|x|^N \left(1-x^2 \right)^{\frac{N-3}{2}}\log|x|dx
\geq -\frac{\mathfrak{Z}_N \cdot \log 2}{2}\cdot (1+\epsilon_N).
\end{gathered}
\end{equation}
The desired result follows from (\ref{eq: entropic tails computation III}) and (\ref{eq: entropic tails computation VI}).
\end{proof}
Finally, we have the following:
\begin{lemma}\label{lem: F_N is not entropically chaotic}
The family of probability densities $\left\lbrace F_{N} \right\rbrace_{N\in\mathbb{N}}$ is not entropically chaotic.
\end{lemma}
\begin{proof}
We saw that $\left\lbrace F_{N} \right\rbrace_{N\in\mathbb{N}}$ is $M_{\frac{1}{2}}-$chaotic so we only need to show that
\begin{equation}\label{eq: F_N is not entropic}
\lim_{N\rightarrow\infty}\frac{H_N(F_N)}{N}\not= H\left(M_{\frac{1}{2}} | \gamma \right).
\end{equation}
Indeed, using symmetry, the monotonicity of the logarithm, equations (\ref{eq: expression for Z_N}) and (\ref{eq: high entropic tails}) we find that
\begin{equation}\label{eq: F_N entropy estimation I}
\begin{gathered}
\frac{H_N(F_N)}{N}=\frac{1}{N\cdot \mathfrak{Z}_N}\int_{\mathbb{S}^{N-1}\left( \sqrt{N} \right)}
\left(\sum_{i=1}^N |v_i|^N \right)\log\left(\sum_{i=1}^N |v_i|^N \right)d\sigma^N \\
-\frac{\log \mathfrak{Z}_N}{N}\geq\frac{1}{\mathfrak{Z}_N}\int_{\mathbb{S}^{N-1}\left( \sqrt{N} \right)}
|v_1|^N\log\left(|v_1|^N \right)d\sigma^N \\
-\frac{(N+2)\log N}{2N}+\frac{(N-1)\log 2}{N} \\
\geq \frac{\log N}{2}-\frac{\log 2}{2}\cdot(1+\epsilon_N)-\frac{(N+2)\log N}{2N}+\frac{(N-1)\log 2}{N}.
\end{gathered}
\end{equation}
Thus
\begin{equation}\label{eq: F_N entropy estimation II}
\liminf_{N\rightarrow\infty}\frac{H_N(F_N)}{N}\geq \frac{\log 2}{2},
\end{equation}
and since $H\left(M_{\frac{1}{2}} | \gamma \right)=\frac{\log 2}{2}-\frac{1}{4}$ our proof is complete.
\end{proof}
\begin{remark}\label{rem: high entropic tails}
Equation (\ref{eq: F_N entropy estimation I}) is exactly why we say that the above example has 'high entropic tails'. The estimation provided in it shows that the rescaled $N-$particle entropy is too high, due to varying power of the polynomial.
\end{remark}
\begin{proof}[Proof of Theorem \ref{thm: 4th counter example}]
This follows immediately from Lemma \ref{lem: F_N,m is entropically chaotic}, \ref{lem: F_N is chaotic} and \ref{lem: F_N is not entropically chaotic}.
\end{proof}
\section{Final Remarks.}\label{sec: final remarks}
While we hope this paper provided a bit of insight into the sensitive nature of entropic chaoticity, there are still many interesting questions on the subject. We present here a few remarks and questions that arose while working on this paper.\\
\begin{itemize}
\item In the examples given in Sections \ref{sec: escaping tensorisation} and \ref{sec: polynomials} we found that both families of probability densities were $M_{\frac{1}{2}}-$chaotic. Since on Kac's sphere we have that
\begin{equation}\label{eq: second moment of the marginal}
1=\frac{1}{N}\int_{\mathbb{S}^{N-1}\left(\sqrt{N} \right)}\left(\sum_{i=1}^N |v_i|^2 \right)F_N(v_1,\dots,v_N)d\sigma^N=\int_{\mathbb{R}}|v_1|^2 \Pi_1\left(F_N\right)(v_1)dv_1,
\end{equation}
and $\int_{\mathbb{R}} |v|^2 M_{\frac{1}{2}}(v)dv=\frac{1}{2}$ something was lost in the limit. This brings the following questions to mind:\\
\textbf{Question:} If a family of probability densities on the sphere, $\left\lbrace F_N \right\rbrace_{N\in\mathbb{N}}$, is $f-$chaotic with $\int_{\mathbb{R}}|v|^2 f(v)dv<1$, can it be entropically chaotic?\\
We believe the answer is negative.\\
\item In light of the above question, one might try and change the dependence in $N$ of the polynomial power in Section \ref{sec: polynomials} to one that will allow convergence without loss of energy. An attempt to pick a power $\alpha_N$ such that $\lim_{N\rightarrow\infty}\frac{\alpha_N}{N}=0$, will not be helpful as it will lead to entropic chaoticity with $\gamma$ as a marginal limit. It seems that $N$ is exactly the power where things break abruptly.\\
\item One can try and replace the definition of entropic chapticity in the case where the limit measure $\mu$ has probability density $f$ with something that might seem more natural. In that case, we define $F_N$ as in (\ref{eq: usual suspect}) (when it makes sense) and say that $\mu_N$ is entropically chaotic if 
\begin{equation}\label{eq: alternative def}
\lim_{N\rightarrow\infty}\frac{H\left(\mu_N | F_N \right)}{N}=0,
\end{equation}
i.e. the rescaled 'distance' between the measure and the intuitive restricted tensorisation of the limit function goes to zero. When $f$ is nice enough (satisfying the conditions of Theorem \ref{thm: usual suspects are good} and a bit more), one can show that the new definition is equivalent to the one we presented here (see \cite{CCRLV,Carr}), however the new definition might be able to deal with infinities more easily and might be less delicate to changes.\\
\textbf{Question:} Are the definitions always equivalent? If not, when and how do they differ?\\
We'd like to point out that in our computable examples the limit function was nice enough to warrant the equivalence of the definitions.
\end{itemize}
The idea of varying functions in accordance to $N$ is the key idea behind many of our constructions and we believe that it is the main way to destroy 'good' properties, or to get horrible decay rates. We believe that such phenomena will not happen if the core function will remain fixed, something that has more of a physical intuition to it, and we're looking forward to follow any advances made on the matter.

\appendix
\section{The Stereographic Projection.}\label{app: stereographic}
The stereographic projection is a way to map $\mathbb{R}^n\cup \{\infty \}$ conformally on $\mathbb{S}^{n}(R)$. The idea is simple: given a point $(x_1,\dots,x_n)\in\mathbb{R}^n$ we can consider it to be a point in $\mathbb{R}^{n+1}$, lying on the hyperplane $x_{n+1}=0$. Connecting it via a straight line to the south pole of $\mathbb{S}^{n+1}(R)$ and intersecting that line with the sphere is the desired map $S(x_1,\dots,x_n)$.\\
In what follows we will find a formula for the stereographic map as well as express the surface element of $\mathbb{S}^{n+1}(R)$ with respect to it. \\
The line connecting the point $(x_1,\dots,x_n,0)$ to the south pole $(0,\dots,0,-R)$ is given by
\begin{equation}\label{eq: line para}
\begin{gathered}
y_i(t)=x_i t \quad i=1,\dots,n.\\
y_{n+1}=-R+Rt. 
\end{gathered}
\end{equation}
Plugging it into the equation of the sphere yields
\begin{equation}\label{eq: intersection I}
\left(\sum_{i=1}^n x_i^2\right)t^2+R^2(1-t)^2=R^2,
\end{equation}
or
\begin{equation}\label{eq: intersection II}
\left(|x|^2+R^2 \right) t^2 -2R^2t=0,
\end{equation}
leading to 
\begin{equation}\label{eq: intersection III}
t=\frac{2R^2}{R^2+|x|^2}.
\end{equation}
Thus
\begin{equation}\label{app: stereographic}
S(x_1,\dots,x_n)=\left(\frac{2R^2 x_1}{R^2+|x|^2},\dots,\frac{2R^2 x_n}{R^2+|x|^2},R\cdot \frac{R^2-|x|^2}{R^2+|x|^2} \right).
\end{equation}
Equation (\ref{app: stereographic}) allows us to find $S^{-1}$ as well. Denoting the variables on $\mathbb{S}^{n+1}(R)$ by $(v_1,\dots,v_{n+1})$ we find that 
\begin{equation}\label{eq: finding inverse I}
v_{n+1}=R\cdot \frac{R^2-|x|^2}{R^2+|x|^2},
\end{equation}
and as such
\begin{equation}\label{app: norm of x^2}
|x|^2=R^2\cdot \frac{R-v_{n+1}}{R+v_{n+1}}.
\end{equation}
Plugging it back into (\ref{app: stereographic}) we find that 
\begin{equation}\label{eq: finding inverse II}
x_i=\frac{\left(R^2+|x|^2\right)v_i}{2R^2}=\frac{Rv_i}{R+v_{n+1}},
\end{equation}
and thus
\begin{equation}\label{app:stereographic inverse}
S^{-1}\left(v_1,\dots,v_{n+1} \right)=\left(\frac{Rv_1}{R+v_{n+1}},\dots,\frac{Rv_n}{R+v_{n+1}} \right).
\end{equation}
In order to express the surface element of the sphere with the $x_i$ coordinates we notice that if $s=S(x)$ and $t=S(y)$ then
\begin{equation}\label{eq: expressing the surface element I}
\begin{gathered}
\sum_{i=1}^{n+1}(s_i-t_i)^2=4R^4\sum_{i=1}^n \left(\frac{x_i}{\left(R^2+|x|^2 \right)}-\frac{y_i}{\left(R^2+|y|^2 \right)}\right)^2
\\+R^2\left(\frac{R^2-|x|^2}{\left(R^2+|x|^2 \right)}-\frac{R^2-|y|^2}{\left(R^2+|y|^2 \right)}\right)^2\\
= \frac{4R^4}{\left(R^2+|x|^2 \right)^2\left(R^2+|y|^2 \right)^2}\Bigg(\left(R^2+|y|^2 \right)^2|x|^2-2\left(R^2+|y|^2 \right)\left(R^2+|x|^2 \right)x\circ y\\
 +\left(R^2+|x|^2 \right)^2|y|^2\Bigg) 
+\frac{R^2\left(2R^2(|y|^2-|x|^2) \right)^2}{\left(R^2+|x|^2 \right)^2\left(R^2+|y|^2 \right)^2}\\
=\frac{4R^4}{\left(R^2+|x|^2 \right)^2\left(R^2+|y|^2 \right)^2}
\Bigg( R^4|x|^2+2R^2|x|^2|y|^2+|y|^4|x|^2\\
 -2\left(R^2+|y|^2 \right)\left(R^2+|x|^2 \right)x\circ y\\
+ R^4|y|^2+2R^2|y|^2|x|^2+|x|^4|y|^2+R^2|y|^4-2R^2|x|^2|y|^2+R^2|x|^4 \Bigg).
\end{gathered}
\end{equation}
Since
\begin{equation}\label{eq: expressing the surface element II}
\left(R^2+|x|^2 \right)\left(R^2+|y|^2 \right)=R^4+R^2|x|^2+R^2|y|^2+|x|^2|y|^2
\end{equation}
we have that 
\begin{equation}\label{eq: computing the differential on the sphere}
\begin{gathered}
\sum_{i=1}^{n+1}(s_i-t_i)^2=\frac{4R^4}{\left(R^2+|x|^2 \right)^2\left(R^2+|y|^2 \right)^2}
\Bigg(\left(R^2+|x|^2 \right)\left(R^2+|y|^2 \right)|x|^2 \\
+\left(R^2+|x|^2 \right)\left(R^2+|y|^2 \right)|y|^2-2\left(R^2+|y|^2 \right)\left(R^2+|x|^2 \right)x\circ y \Bigg)\\
=\frac{4R^4|x-y|^2}{\left(R^2+|x|^2 \right)\left(R^2+|y|^2 \right)},
\end{gathered}
\end{equation}
from which we conclude that the metric on the sphere is given by
\begin{equation}\label{app: metric on the sphere}
ds_R=\left(\frac{2R^2}{R^2+|x|^2} \right)^n dx_1 \dots dx_n.
\end{equation}

\end{document}